\documentclass[10pt]{article}
\usepackage{amsfonts, epsfig, amsmath, amssymb,amsthm, color,mathabx}
\usepackage{mathrsfs}

\usepackage[english]{babel}

\usepackage[round]{natbib}   

\let\oldmarginpar\marginpar
\renewcommand{\marginpar}[1]{\oldmarginpar{\scriptsize\texttt{\color{red}{#1}}}}

\newcommand{\E}{\mathbf{E}}
\renewcommand{\P}{\mathbf{P}}

\renewcommand {\epsilon}{\varepsilon}

\newtheorem{thm}{Theorem}[section]
\newtheorem{prp}[thm]{Proposition}
\newtheorem{cor}[thm]{Corollary}
\newtheorem{lem}[thm]{Lemma} 

\theoremstyle{definition}

\newtheorem{exa}[thm]{Example} 
\newtheorem{rem}[thm]{Remark}

\DeclareMathSymbol{\ophi}{\mathalpha}{letters}{"1E}

\newcommand{\e}{\varepsilon}

\renewcommand{\phi}{\varphi}

\newcommand{\be}{\begin{equation}}
\newcommand{\ee}{\end{equation}}
\newcommand{\ben}{\begin{equation*}}
\newcommand{\een}{\end{equation*}}

\newcommand{\ba}{\begin{equation}\begin{aligned}}
\newcommand{\ea}{\end{aligned}\end{equation}}

\DeclareMathOperator{\Law}{Law}

\renewcommand{\i}{\mathrm{i}}
\newcommand{\ex}{\mathrm{e}}
\newcommand{\di}{\mathrm{d}}

\newcommand{\cN}{\mathcal{N}}

\newcommand{\rB}{\mathscr{B}}

\newcommand{\rF}{\mathscr{F}}
\newcommand{\rG}{\mathscr{G}}

\newcommand{\bD}{\mathbb{D}}
\newcommand{\bF}{\mathbb{F}}
\newcommand{\bI}{\mathbb{I}}

\newcommand{\bR}{\mathbb{R}}

\newfont{\cyrfnt}{wncyr10}
\def\J3{\cyrfnt{\rm \u{\cyrfnt I}}}
\def\j3{\cyrfnt{\rm \u{\cyrfnt i}}}

\allowdisplaybreaks[4]

\numberwithin{equation}{section}

\begin{document}
\title{Drift Estimation for a L\'evy-Driven Ornstein-Uhlenbeck Process with Heavy Tails}

\date{\today}


\author{Alexander Gushchin\footnote{Steklov Mathematical Institute of Russian Academy of Sciences, 8 Gubkina St., Moscow,
119991, Russia; 
gushchin@mi-ras.ru} \footnote{Lomonosov Moscow State University, Moscow, Russia} \footnote{National Research University Higher School of Economics, Moscow, Russia},\ \  Ilya Pavlyukevich\footnote{Institute of Mathematics, Friedrich Schiller University Jena, Ernst--Abbe--Platz 2,
07743 Jena, Germany; ilya.pavlyukevich@uni-jena.de},\ \ and\ Marian Ritsch\footnote{Institute of Mathematics, 
Friedrich Schiller University Jena, Ernst--Abbe--Platz 2,
07743 Jena, Germany; carl.christian.marian.ritsch@uni-jena.de}}
\maketitle

\begin{abstract}
We consider the problem of estimation of the drift parameter of an 
ergodic Ornstein--Uhlenbeck type process driven by a L\'evy process with heavy tails. The process is observed continuously on a long time 
interval $[0,T]$, $T\to\infty$. We prove that the statistical model is locally asymptotic mixed normal and the maximum
likelihood estimator is asymptotically efficient.
\end{abstract}


\smallskip


\smallskip


\section{Introduction, motivation, previous results}

In his paper, we deal with an estimation of the drift parameter $\theta>0$ of an ergodic 
one-dimensional Ornstein--Uhlenbeck process $X$ driven by a L\'evy process:
\ba
X_t=X_0-\theta\int_0^t X_s\,\di s+ Z_t,\quad t\geq 0.
\ea
The process $Z$ is a one-dimensional L\'evy process with known characteristics and 
with infinite variance.
The process $X$ is observed continuously on a long time interval $[0,T]$, $T\to\infty$. The problem is to study asymptotic properties
of the corresponding statistical model and to show that the maximum likelihood estimator of $\theta$ is asymptotically efficient in an
appropriate sense. Although the continuous time observations are far from being realistic in applications, they are of theoretical 
importance since they can be considered as a limit of high frequency discrete models.

Since we deal with continuous observations, it is natural to assume that the Gaussian component of the L\'evy 
process $Z$ is not degenerate. In this case, the laws of observations corresponding to different values of $\theta$
are equivalent and the likelihood ratio has an explicit form.

There are a lot of papers devoted to inference for L\'evy driven SDEs. Most of the literature treats the case of discrete time 
observations both in the high and low frequency setting. A general theory for the 
likelihood inference for continuously observed jump-diffusions can be found in \cite{Soerensen-91}.

A complete analysis of the drift estimation for 
continuously observed ergodic and non-ergodic Ornstein--Uhlenbeck process driven by a Brownian motion 
can be found in \cite[Chapter 8.1]{hoepfner2014asymptotic}.

For continuously observed square integrable L\'evy driven Ornstein--Uhlenbeck processes,
the local asymptotic normality (LAN) of the model and the asymptotic efficiency of the maximum likelihood estimator
of the drift
have been derived by \cite{mai2012diss,Mai-14} 
with the help of the theory of exponential families, see \cite{KuechlerS-97}.

High frequency estimation of a square integrable L\'evy driven Ornstein--Uhlenbeck process with non-vanishing
Gaussian component has been performed by \cite{mai2012diss,Mai-14}.
\cite{kawai2013local} studied the asymptotics of the Fisher information for three characterizing parameters of Ornstein--Uhlenbeck
processes with jumps under low frequency and high frequency discrete sampling. The existence of all moments of the L\'evy process 
was assumed.
\cite{tran2017lan} considered  the ergodic Ornstein--Uhlenbeck process
driven by a Brownian motion and a compensated Poisson process, whose drift and diffusion
coefficients as well as its jump intensity depend on unknown parameters.
He obtained the LAN property of the model in the high frequency setting.
 
We also mention the works by \cite{hu2007parameter,hu2009least,yaozhong2009singularity,long2009least,zhang2013least}
devoted to the least-square estimation of parameters of the Ornstein--Uhlenbeck process driven by an $\alpha$-stable L\'evy process.

There is vast literature devoted to parametric inference for discretely observed 
L\'evy processes (see, e.g.\ a survey by \cite{masuda2015parametric}) and L\'evy driven SDEs. 
More results on the latter topic can be found e.g.\ in
\cite{masuda2013convergence,ivanenko2014lan,kohatsu2017lan,masuda2019non,uehara2019statistical,clement2015local,clement2019lamn,ClementG-19,nguyen2018estimation,gloter2018jump}
and the references therein.

\medskip

In this paper, we fill the gap and analyse continuously observed ergodic Ornstein--Uhlenbeck process driven
by a L\'evy process with heavy regularly varying tails of the index $-\alpha$, $\alpha\in (0,2)$, in the presence of a Gaussian component. 
It turns out that the log-likelihood in this model is quadratic, however the model is not
asymptotically normal and we prove only the local asymptotic mixed normality (LAMN) property.
We refer to \cite{lecam2000,hoepfner2014asymptotic} for the general theory of estimation for LAMN models.

The fact that the prelimiting log-likelihood is quadratic automatically implies that the maximum likelihood estimator is asymptotically efficient in the sense
of Jeganathan's convolution theorem and attains the local asymptotic minimax bound.
Another feature of our model is that the asymptotic observed information has spectrally positive $\alpha/2$-stable distribution. 
This implies that the limiting law of the maximum likelihood estimator has tails of the order $\exp(-x^\alpha)$ and hence finite moments of all orders. 

\medskip

\noindent  
\textbf{Acknowledgements:} The authors thank the DAAD exchange programme \emph{Eastern Partnership} for financial support. 
A.G.\ thanks Friedrich Schiller University Jena for hospitality.

\section{Setting and the main result}

Consider a stochastic basis $(\Omega,\rF,\bF,\P)$, $\bF$ being right-continuous. 
Let $Z$ be a L\'evy process with the characteristic triplet $(\sigma^2,b,\nu)$ and the L\'evy--It\^o decomposition
\ba
Z_t=\sigma W_t+ b t+ \int_0^t \int_{|z|\leq 1}z\tilde N(\di z,\di s)
+  \int_0^t \int_{|z|> 1}z  N(\di z,\di s) ,
\ea
where $W$ is a standard one-dimensional Brownian motion, $N$ is a Poissonian random measure on $\bR\backslash\{0\} $ with the L\'evy measure $\nu$
satisfying $\int_\bR (z^2\wedge 1)\,\nu(\di z)<\infty$, 
 $\tilde N$ is the compensated 
Poissonian random measure, and $b\in\bR$.

For $\theta\in \bR$, let $X$ be an Ornstein--Uhlenbeck type process being a solution of the SDE
\ba
\label{e:X}
X_t=X_0-\theta\int_0^t X_s\,\di s +  Z_t ,\quad t\geq 0,
\ea
where $\theta\in\bR$ is an unknown parameter.
The initial value $X_0\in\rF_0$ is a random variable whose distribution does not depend on $\theta$.
Note that $X$ has an explicit representation
\ba
X_t=X_0\ex^{-\theta t}-\int_0^t \ex^{-\theta(t-s)}\,\di Z_s,\quad t\geq 0,
\ea
see, e.g.\ \cite[Sections 4.3.5 and 6.3]{Applebaum-09} and \cite[Section 17]{Sato-99}.

Let $\mathbb D=D([0,\infty),\bR)$ be the space of real-valued c\`adl\`ag functions $\omega\colon [0,\infty)\to\bR$ 
equipped with Skorokhod topology and Borel $\sigma$-algebra $\rB(\bD)$. The space
$(\bD,\rB(\bR))$  is Polish, and $\rB(\bD)$
coincides with the $\sigma$-algebra generated by the coordinate projections.
We define a (right-continuous) filtration $\mathbb G=(\rG_t)_{t\geq 0}$ consisting of $\sigma$-algebras
\ba
&\rG_t:=\bigcap_{s> t} \sigma\Big(\omega_r\colon  r\leq s, \omega\in \bD\Big),\quad t\geq 0.
\ea
For each $\theta\in\bR$, the process $X=(X_t)_{t\geq 0}$ induces a measure $\P^\theta$ on the path space $(\bD,\rB(\bD))$. 
Let
\ba
\P^\theta_T=\P^\theta\Big|_{\rG_T}
\ea
be a restriction of $\P^\theta$ to the $\sigma$-algebra $\rG_T$. 

In order to establish the equivalence of the laws $\P^\theta_T$ and $\P^{\theta_0}_T$, $\theta,\theta_0\in \bR$, we have to make the 
following assumption.

\smallskip

\noindent
\textbf{A}$_\sigma$:
The Brownian component of $Z$ is non-degenerate, i.e.\ $\sigma>0$.

\begin{prp}
\label{p:Lratio}
Let \emph{\textbf{A}}$_\sigma$ hold true. Then for each $T>0$, any $\theta,\theta_0\in\bR$
\ba
&\P^\theta_T\sim \P^{\theta_0}_T,
\ea
and the likelihood ratio is given by
\ba
&L_T(\theta_0,\theta)=\frac{\di \P^{\theta}_T}{\di \P^{\theta_0}_T}=
\exp\Big(-\frac{\theta-\theta_0}{\sigma^2} \int_0^T \omega_s\,\di m_s^{(\theta_0)} 
-\frac{(\theta-\theta_0)^2}{2\sigma^2}\int_0^T \omega_s^2\,\di s  \Big),
\ea
where 
\ba
m^{(\theta_0)}_t=\omega_t-\omega_0+\theta_0 \int_0^t \omega_s\, \di s 
-bt
-\sum_{s\leq t} \Delta \omega_s\bI(|\Delta \omega_s|> 1)
- \int_0^t\int_{|x|\leq 1} x \Big(\mu(\di x,\di s)  -\nu(\di x)\di s\Big) 
\ea
is the continuous local martingale component of $\omega$ under the measure $\P^{\theta_0}_T$,
and the random measure 
\ba
\mu(\di x,\di s)=\sum_{s} \bI(\Delta \omega_s\neq 0)\delta_{(\Delta \omega_s, s)}(\di x,\di s)
\ea
is defined by the jumps of $\omega$.
\end{prp}
\begin{proof}
See \cite[Theorem III-5-34]{JacodS-03}.
\end{proof}

Consider a family of statistical experiments
\ba
\label{e:se}
\Big(\bD, \rG_T,\{\P^\theta_T\}_{\theta>0}\Big)_{T>0}.
\ea
Our goal is to establish local asymptotic mixed normality (LAMN) of these experiments
under the assumption that 
the process $Z$ has heavy tails. We make the following assumption.

\smallskip

\noindent
\textbf{A}$_\nu$:
The L\'evy measure $\nu$ has a regularly varying heavy tail of the order $\alpha\in (0,2)$, i.e.\
\ba
H(R):=\int_{|z|> R} \nu(\di z)\in \text{RV}_{-\alpha},\quad R>0.
\ea
In other words, 
$H\colon (0,\infty)\to (0,\infty)$ and there is a
positive function $l=l(R)$ slowly varying at infinity such that  
\ba
H(R)=\frac{l(R)}{R^{\alpha}},\quad R>0.
\ea

For the tail $H$ we construct an \emph{absolutely continuous strictly decreasing} function 
\ba
\tilde H(R)=\alpha\int_R^\infty \frac{H(z)}{z}\,\di z,\quad R>0,
\ea
such that by Karamata's theorem, see e.g.\ \cite[Theorem 2.1 (a)]{resnick2007heavy},
\ba
\lim_{R\to\infty}\frac{\tilde H(R)}{H(R)}=1.
\ea
We introduce the monotone increasing continuous scaling $\{\phi_T\}_{T>0}$ defined by the relation
\ba
\label{e:phi}
\frac{1}{\phi_T}:=\tilde H^{-1}\Big(\frac{1}{T}\Big),
\ea
where $\tilde H^{-1}(R):=\inf\{u>0\colon \tilde H(u)=R\}$ is the (continuous) inverse of $\tilde H$. It is easy to see that
$\phi_T\in \text{RV}_{-1/\alpha}$.

\begin{rem}
We make use of the absolutely continuous strictly decreasing function $\tilde H$ just for convenience in order to 
avoid technicalities connected with the inversion of c\`adl\`ag functions. For instance it holds
$\phi_T\sim \big(H^{\leftarrow}(1/T)\big)^{-1}$ for the generalized inverse $H^\leftarrow(R):=\inf\{u>0\colon H(u)>R\}$, see 
\cite[Chapter 1.5.7]{BinghamGT-87}.
\end{rem}

\begin{exa}
 Let the jump part of the process $Z$ be an $\alpha$-stable L\'evy process, i.e.\ for $\alpha\in (0,2)$ and $c_-,c_+\geq 0$, $c_- +c_+>0$, let
\ba
\nu(\di z)=\Big(\frac{c_-}{|z|^{1+\alpha}}\bI(z<0)+\frac{c_+}{z^{1+\alpha}}\bI(z>0)\Big)\,\di z.
\ea
Then
\ba
H(R)&=\tilde H(R)=\frac{c_-+c_+}{\alpha R^\alpha},\\
\tilde H^{-1}(T)&= \Big(\frac{\alpha}{c_-+c_+}\Big)^{1/\alpha}\frac{1}{T^{1/\alpha}},
\ea
and 
\ba
\phi_T=\Big(\frac{c_-+c_+}{\alpha}\Big)^{1/\alpha}\frac{1}{T^{1/\alpha}}.
\ea
\end{exa}
The main result is the LAMN property of our model.
\begin{thm}
\label{t:main}
Let \emph{\textbf{A}}$_\sigma$ and \emph{\textbf{A}}$_\nu$ hold true. Then 
the family of statistical experiments \eqref{e:se} is locally asymptotically mixed normal at each $\theta_0>0$, namely
for each $u\in\bR$
\ba
\Law\Big(\ln L(\theta_0,\theta_0+\phi_T u)\Big|\P^{\theta_0}_T\Big)\to \cN\sqrt{\frac{\mathcal S^{(\alpha/2)}}{2\sigma^2\theta_0}} u 
- \frac12 \frac{\mathcal S^{(\alpha/2)}}{2\sigma^2\theta_0}u^2  ,\quad T\to\infty,
\ea
where $\cN$ is a standard Gaussian random variable and $\mathcal S^{(\alpha/2)}$ is an independent spectrally positive $\alpha/2$-stable random variable
with the Laplace transform
\ba
\label{e:LT}
\E \ex^{-\lambda \mathcal S^{(\alpha/2)} }=\ex^{-\Gamma(1-\frac{\alpha}{2})\lambda^{\alpha/2} },\quad \lambda\geq 0. 
\ea
\end{thm}

\noindent 
Theorem \ref{t:main} is based on the following key result.
\begin{thm}
\label{t:X2}
Let \emph{\textbf{A}}$_\sigma$ and \emph{\textbf{A}}$_\nu$ hold true. Then for each $\theta_0>0$
\ba
\Law\Big(\phi_T^2 \int_0^T X_s^2\,\di s\Big|\P^{\theta_0}_T\Big)\to  \frac{\mathcal S^{(\alpha/2)}}{2\theta_0},
\ea
where $\mathcal S^{(\alpha/2)}$ is a random variable with the Laplace transform \eqref{e:LT}.
\end{thm}

\begin{cor}
\label{t:XWX2}
Let \emph{\textbf{A}}$_\sigma$ and \emph{\textbf{A}}$_\nu$ hold true. Then for each $\theta_0>0$
\ba
\Law \Big(  \phi_T \int_0^T X_s\,\di W_s,  \phi_T^2 \int_0^T X_s^2\,\di s\Big|\P^{\theta_0}_T\Big)
\to  \Big(\cN \sqrt{\frac{\mathcal S^{(\alpha/2)}}{2\theta_0}}  ,      \frac{\mathcal S^{(\alpha/2)}}{2\theta_0}\Big).
\ea
\end{cor}

Proposition \ref{p:Lratio} and Theorem \ref{t:main} allow us to establish asymptotic 
distribution of the maximum likelihood estimator $\hat \theta_T$ of $\theta$. Moreover, the special form of the likelihood ratio
guarantees that $\hat\theta_T$ is asymptotically efficient.

\begin{cor}
\label{c:2}
1. Let \emph{\textbf{A}}$_\sigma$ hold true. Then the maximum likelihood estimator $\hat \theta_T$ of $\theta$ 
satisfies
\ba
\label{e:MLE}
\hat \theta_T=\theta_0-\frac{\int_0^T \omega_s\,\di \omega^c_s}{ \int_0^T  \omega_s^2\,\di s}.
\ea
2. Let \emph{\textbf{A}}$_\sigma$ and \emph{\textbf{A}}$_\nu$ hold true. Then 
\ba
\label{e:AB}
\Law\Big(\frac{\hat \theta_T-\theta_0}{\phi_T}\Big|\P^{\theta_0}_T\Big)
\to\sigma\sqrt{2\theta_0}\cdot \frac{\cN }{\sqrt{ \mathcal S^{(\alpha/2)}}},\quad T\to\infty.
\ea
The maximum likelihood estimator $\hat \theta_T$ is asymptotically efficient in the sense of the convolution 
theorem and the local asymptotic minimax theorem for LAMN models, see \cite[Theorems 7.10 and 7.12]{hoepfner2014asymptotic}.
\end{cor}

\begin{rem}
\label{r:D}
It is instructive determine the tails of the random variable $\cN/\sqrt{\mathcal S^{(\alpha/2)}}$: for each $\alpha\in (0,2)$
\ba
\label{e:tail}
\limsup_{x\to+\infty} x^{-\alpha}\ln \P\Big(\frac{|\cN|}{\sqrt{\mathcal S^{(\alpha/2)}}}> x\Big)<0,
\ea
and in particular
all moments of the r.h.s.\ of \eqref{e:AB} are finite.
\end{rem}

\medskip

The proof of all the results formulated above will be given in Section \ref{s:proofs} after necessary preparations made in the next Section.

\section{Auxiliary results}

We decompose the L\'evy process $Z$ into a compound Poisson process with heavy jumps, and the rest. For definiteness, let for $\rho\geq 0$,
$R_T=T^\rho\colon [1,\infty)\to [1,\infty)$ be a non-decreasing function.

Denote
\ba
\eta^T_t&=\int_0^t \int_{|z|>R_T} zN(\di z,\di s),\\
\xi_t^T&=\sigma W_t   +\int_0^t \int_{|z|\leq R_T} z\tilde N(\di z,\di s),\\
b_T&=b+\int_{1< |z|\leq R_T} z\nu(\di z),\\
Z_t^T&=Z_t-\eta^T_t=  \xi^T_t + b_T t.
\ea
For each $T\geq  1$, the process $\eta^T$ is a compound Poisson process
with intensity $H(R_T)$, the iid jumps $\{J^T_k\}_{k\geq 1}$ occurring at arrival times $\{\tau_k^T\}_{k\geq 1}$, such that
\ba
&\P(|J_k^T|\geq z)=\frac{H(z)}{H(R_T)},\quad z\geq R_T,\\
&\P(\tau_{k+1}^T-\tau_k^T>u)=\ex^{-  H(R_T)u  },\quad u\geq 0.
\ea
Denote also by $N^T$ the Poisson counting process of $\eta^T$; it is a Poisson process with intensity  $H(R_T)$.

We decompose the Ornstein--Uhlenbeck process $X$ into a sum
\ba
\label{e:XX}
X_t&=X^T_t+X^{\eta^T}_t,\\
X_t^T&:= X_0\ex^{-\theta t}+ \int_0^t \ex^{-\theta(t-s)}\,\di Z^T_s,\\
X_t^{\eta^T}&:= \int_0^t \ex^{-\theta(t-s)}\,\di \eta^T_s.
\ea

Since $H(\cdot)\in\text{RV}_{-\alpha}$ and $\phi_\cdot\in\text{RV}_{-1/\alpha}$, $\alpha\in (0,2)$, by Potter's bounds 
(see, e.g.\ \cite[Proposition 2.6 (ii)]{resnick2007heavy})
for each $\e>0$ there are constants
$0<c_\e\leq C_\e<\infty$ such that for $u\geq 1$
\ba
\label{e:est-eps}
\frac{c_\e}{u^{\alpha+\e}}&\leq  H(u)\leq \frac{C_\e}{u^{\alpha-\e}},\\
\frac{c_\e}{u^{\frac{1}{\alpha}+\e}} &\leq  \phi_u\leq \frac{C_\e}{u^{\frac{1}{\alpha}-\e}}.
\ea
The following Lemma gives useful asymptotics of the truncated moments of the L\'evy measure $\nu$.
\begin{lem}
\label{l:moments}
1. For $\alpha\in (0,1]$ and any $\e>0$ there is $C(\e)>0$ such that
\ba
\label{e:int1}
&\int_{1<|z|\leq R} |z|\nu(\di z)\leq C(\e)R^{1-\alpha+\e}.
\ea
2. For $\alpha\in (1,2)$ there is $C>0$ such that
\ba
\label{e:int2}
&\int_{1<|z|\leq R} |z|\nu(\di z)\leq C.
\ea
3. For $\alpha\in (0,2)$ and any $\e>0$ there is $C(\e)>0$ such that
\ba
\label{e:int3}
&\int_{1<|z|\leq R} z^2 \nu(\di z)\leq C(\e)R^{2-\alpha+\e}.
\ea
\end{lem}
\begin{proof}
To prove the first inequality we integrate by parts and note that for any $\e>0$
\ba
\int_{1<|z|\leq R} |z|\nu(\di z) = - \int_{(1,R]} z\, \di H(z) 
&=- z H(z)\Big|_{1}^R  +   \int_{(1,R]} H(z)\,\di z\\
&\leq H(1) + C_\e \int_1^R \frac{\di z}{z^{\alpha-\e}}.
\ea
Hence \eqref{e:int1} follows for any $\e>0$ and \eqref{e:int2} is obtained if we choose $\e\in(0,\alpha -1)$.
The estimate \eqref{e:int3} is obtained analogously to  \eqref{e:int1}.
\end{proof}

The next Lemma will be used to determine the tail behaviour of the product of any two independent normalized 
jumps $|J_k^T||J_l^T|/R^2_T$, $k\neq l$.
\begin{lem}
\label{l:UV}
Let $U_R\geq 1$ and $V_R\geq 1$ be two independent random variables with the probability distribution function
\ba
\P(U_R>x)=\P(V_R>x)=\bar F_R(x)=\frac{H(xR)}{H(R)},\quad R\geq 1,\quad x\geq 1.
\ea
Then for each $\e\in (0,\alpha)$ there is $C(\e)>0$ such that
for all $R\geq 1$ and all $x\geq 1$
\ba
\P(U_RV_R> x)\leq \frac{C(\e)}{x^{\alpha-\e}}.
\ea
\end{lem}
\begin{proof}
Recall that Potter's bounds \cite[Proposition 2.6 (ii)]{resnick2007heavy} imply that for each $\e>0$ there is $C_0(\e)>0$ such that 
for each $x\geq 1$ and $R\geq 1$
\ba
\bar F_R(x)=\frac{H(xR)}{H(R)}\leq \frac{C_0(\e)}{x^{\alpha-\e}}.
\ea
Moreover,
\ba
\bar F_R (x)&\equiv 1,\quad x\in[0,1].
\ea
For $x> 1$ we write
\ba
\P(U_RV_R>x)&=\int_1^\infty \int_{x/u}^\infty \di F_R(v)\,\di F_R(u)\\
&=\Big(\int_1^{x}+\int_{x}^\infty\Big) \bar F_R(x/u)\,\di F_R(u)=I^{(1)}_{R}(x)+I^{(2)}_{R}(x) .
\ea
Then
\ba
I^{(2)}_{R}(x)&= \int_{x}^\infty\bar F_R(x/u)\,\di F_R(u)\leq \int_{x}^\infty \di F_R(u)
\leq \bar F_R(x)\leq \frac{C_0(\e)}{x^{\alpha-\e}}.
\ea
Eventually,
\ba
I^{(1)}_{R}(x)&\leq  \frac{C_0(\e)}{x^{\alpha-\e}}   \int_1^{x}  u^{\alpha-\e} \,\di  F_R(u)
=- \frac{C_0(\e)}{x^{\alpha-\e}}   \int_1^{x}  u^{\alpha-\e}     \,\di \bar F_R(u)\\
&=- \frac{C_0(\e)}{x^{\alpha-\e}} u^{\alpha-\e} \bar F_R(u)\Big|_1^{x}
+    (\alpha-\e)\frac{C_0(\e)}{x^{\alpha-\e}} \int_1^{x}  u^{\alpha-1-\e}  \bar F_R(u)    \,\di u\\
&\leq \frac{C_0(\e)}{x^{\alpha-\e}}
+ (\alpha-\e)\frac{C_0(\e)^2}{x^{\alpha-\e}} \int_1^{x}   \frac{u^{\alpha-1-\e}}{u^{\alpha-\e}}    \,\di u\\
&\leq \frac{C_0(\e)}{x^{\alpha-\e}}
+ (\alpha-\e)\frac{C_0(\e)^2}{x^{\alpha-\e}} \ln x\\
&\leq \frac{C(\e)}{x^{\alpha-2\e}}
\ea 
for some $C(\e)>0$. 
\end{proof}

\begin{rem}
A finer tail asymptotics of products of iid non-negative Pareto type random variables can be found in 
\cite[Theorem 2.1]{rosinski1987multilinear} and \cite[Lemma 4.1 (4)]{jessen2006regularly}. In Lemma \ref{l:UV}, however, we establish 
rather rough estimates which are valid for the families of iid random variables $\{U_R,V_R\}_{R\geq 1}$.
\end{rem}

The following useful Lemma will be used to determine 
the conditional distribution of the interarrival times of the compound Poisson process $\eta^T$. 
\begin{lem}
Let $T>0$ and
let $N=(N_t)_{t\in[0,T]}$ be a Poisson process, $\{\tau_k\}_{k\geq 1}$ be its arrival
times, $\tau_0=0$. Then for each $m\geq 1$, and $1\leq j < j+k\leq m$
\ba
\label{e:beta}
\P(\tau_{j+k}-\tau_j\leq s| N_T=m)=\P\Big(\sigma_k\leq \frac{s}{T}  \Big),\quad s\in[0,1],
\ea
where 
$\sigma_k$ is a $\operatorname{Beta}(m,k-1)$-distributed random variable with the density
\ba
\label{e:betadensity}
f_{\sigma_k}^{(m)}(u)=\frac{m!}{(k-1)!(m-k)!}u^{k-1}(1-u)^{m-k},\quad u\in[0,1],\quad m\geq 1,\ 1\leq k\leq m.
\ea
\end{lem}
\begin{proof}
It is well known that the conditional distribution of the arrival times $\tau_1 ,\dots , \tau_m$,
given that $N_T = m$, coincides with the distribution of the order statistics obtained from
$m$ samples from the population with uniform distribution on $[0, T]$, see \cite[Proposition 3.4]{Sato-99}.

Let for brevity $T=1$. The joint density of $(\tau_j,\tau_{j+k})$, $1\leq j<j+k\leq m$ is well known, see e.g.\ 
\cite[Chapter 11.10]{BalNev-03}:
\ba
f_{\tau_j,\tau_{j+k}}^{(m)}(u,v)&=c_{j,k,m} \cdot u^{j-1}(v-u)^{k-1}(1-v)^{m-j-k}\bI(0\leq u<v\leq 1),\\
c_{j,k,m}&=\frac{m!}{(j-1)!(k-1)!(m-j-k)!},
\ea
and consequently
\ba
f^{(m)}_{\tau_{j+k}-\tau_{j}, \tau_{j} }(u,v)=c_{j,k,m}\cdot v^{j-1}u^{k-1}(1-u-v)^{m-j-k},\quad u,v,u+v\in[0, 1].
\ea
Hence, the probability density of the difference $\tau_{j+k}-\tau_{j}$ is obtained by integration w.r.t.\ $v\in[0,1]$,
\ba
f^{(m)}_{\tau_{j+k}-\tau_{j}}(u)&=c_{k,j,m}\cdot u^{k-1}\int_0^{1-u} v^{j-1}(1-u-v)^{m-j-k}\,\di v \\ 
&\stackrel{v=(1-u)z}{=}c_{j,k,m}\cdot u^{k-1} \cdot (1-u)^{m-k}  \int_0^1 z^{j-1}(1-z)^{m-j-k}\,\di z.
\ea
Recalling the definition of the Beta-function, we get
\ba
\int_0^1 z^{j-1}(1-z)^{m-j-k}\,\di z=\frac{(j-1)!(m-j-k)!}{(m-k)!},
\ea
which yields the desired result.
\end{proof}

\begin{lem}
\label{l:Slaplace}
Let \emph{\textbf{A}}$_\nu$ hold true and $\{\phi_T\}$ be the scaling defined in \eqref{e:phi}. Then for any 
$\rho\in[0,\frac{1}{\alpha})$ 
\ba
\phi_T^2 [\eta^T]_T\stackrel{\di}{\to} \mathcal S^{(\alpha/2)},\quad T\to \infty,
\ea
where $\mathcal S^{(\alpha/2)}$ is a spectrally positive $\alpha/2$-stable random variable with the Laplace transform \eqref{e:LT}.
\end{lem}
\begin{proof}
The process $t\mapsto \phi_T^2 [\eta^T]_t$ is a compound Poisson process with the L\'evy measure $\nu_T$ with the tail
\ba
H_T(u)=\int_u^\infty \nu_T(\di z)= H\Big(\frac{\sqrt u}{\phi_T}\vee R_T\Big) ,\quad u>0.
\ea
The Laplace transform of $\phi_T^2 [\eta^T]_T$ has the cumulant
\ba
K_T(\lambda):=\ln \E \ex^{- \lambda \phi_T^2 [\eta^T]_T}
&=-T\int_0^\infty \Big(\ex^{-\lambda u}-1\Big)\,\di H_T(u),\quad \lambda\geq 0.
\ea
Integrating by parts yields
\ba
\label{e:K}
K_T(\lambda)=-T\big(\ex^{-\lambda u}-1\big) H_T(u)\Big|_{0}^\infty 
-\lambda T \int_0^\infty \ex^{-\lambda u}H_T(u)\,\di u.
\ea
Since the first summands on the r.h.s.\ of \eqref{e:K} vanish, it is left to evaluate the integral term.
Taking into account \eqref{e:phi}, namely that $\frac{1}{T}=\tilde H(\frac{1}{\phi_T})$, we write for any $u_0>0$
\ba
K_T(\lambda)&=
-\lambda T\int_0^\infty  \ex^{-\lambda u} H_T(u)\,\di u\\
&=-\lambda T \int_{0}^{u_0} \ex^{-\lambda u}  H\Big(\frac{\sqrt u}{\phi_T}\vee R_T\Big)    \,\di u
-\lambda
\frac{H(1/\phi_T)}{\tilde H(1/\phi_T)}\frac{1}{H(1/\phi_T)}
\int_{u_0}^\infty \ex^{-\lambda u}  H\Big(\frac{\sqrt u}{\phi_T}\vee R_T\Big)    \,\di u\\
&=-I_T^{(1)}(\lambda)-I_T^{(2)}(\lambda).
\ea
It is evident that $\lim_{T\to\infty}\frac{H(1/\phi_T)}{\tilde H(1/\phi_T)}=1$. Moreover for $\rho\in[0,1/\alpha)$
due to \cite[Proposition 2.4]{resnick2007heavy},
the convergence 
\ba
\lim_{T\to\infty}  \frac{H\big(\frac{\sqrt u}{\phi_T}\vee R_T\big)}{ H\big(\frac{1}{\phi_T}\big)  }  =   \frac{1}{u^{\alpha/2}}
\ea
holds uniformly on each half-line $[u_0,\infty)$, $u_0>0$, and thus for each $u_0>0$
\ba 
\lim_{T\to\infty} I_T^{(2)}( \lambda)=\lambda \int_{u_0}^\infty \frac{\ex^{-\lambda u}  }{u^{\alpha/2}} \, \di u.
\ea
Further we estimate
\ba
I^{(1)}_T(\lambda)
\leq 2\lambda  T \phi_T^2 \int_{0}^{\sqrt{u_0}/\phi_T}  y H (y)\, \di y.
\ea
Note that $y\mapsto yH(y)$ is integrable at $0$ by the definition of the L\'evy measure, $0\leq -\int_0^1 y^2\di H(y)<\infty$, 
and the integration by parts.
Eventually by Karamata's theorem \cite[Theorem 2.1 (a)]{resnick2007heavy}
\ba
I_T^{(1)}\leq 2 \lambda \frac{H(1/\phi_T)}{\tilde H(1/\phi_T)}\cdot \frac{\phi_T^2 }{H(1/\phi_T)}\cdot 
\frac{\int_{0}^{\sqrt{u_0}/\phi_T}  y H (y)\, \di y}{ \frac{u_0}{\phi_T^2} H(\frac{\sqrt{u_0}}{\phi_T})  } 
\cdot \frac{u_0}{\phi_T^2}\cdot H\Big(\frac{\sqrt{u_0}}{\phi_T}\Big) \to \frac{2\lambda }{2-\alpha}u_0^{1-\frac{\alpha}{2}},\quad T\to\infty.
\ea
Hence choosing $u_0>0$ sufficiently small and letting $T\to\infty$ we obtain convergence of $K_T$ to the cumulant of a
spectrally positive stable random variable
\ba
\lim_{T\to\infty} K_T(\lambda)= -\lambda \int_{0}^\infty \frac{\ex^{-\lambda u}  }{u^{\alpha/2}} \, \di u= 
-\Gamma\Big(1-\frac{\alpha}{2}\Big)\lambda^{\alpha/2}.
\ea
\end{proof}

\begin{lem}
\label{l:OUconv}
For any $\rho\in [0,1/\alpha)$ and any $\theta>0$
\begin{align}
\label{e:XT}
& \phi_T^2 |X^{T}_T|^2 \stackrel{\di}{\to } 0,\quad T\to\infty,\\
& \phi_T^2 \int_0^T |X^{T}_s|^2\,\di s \stackrel{\di}{\to } 0,\quad T\to\infty.
\end{align}
\end{lem}
\begin{proof}
\ba
|X^{T}_s|^2\leq 3|X_0|^2\ex^{-2\theta s}+ 3\ex^{-2\theta s}\Big|\int_0^s \ex^{\theta r}\,\di \xi^T_r\Big|^2
+3 b_T^2 \ex^{-2\theta s}\Big|\int_0^s \ex^{\theta r}\,\di r\Big|^2= a_1(s)+a_2(s)+a_3(s).
\ea
By the It\^o isometry and Lemma \ref{l:moments}, for any $\e>0$ we estimate for each $s\geq 0$
\ba
\label{e:xi}
\phi_T^2 \E a_2(s) 
&=\phi_T^2 \cdot  \frac{3}{2\theta}\Big(\sigma^2 + \int_{|z|\leq R_T} z^2\nu(\di z)\Big) \cdot \ex^{-2\theta s} 
  (\ex^{2\theta s}-1) 
\leq C_1 \cdot  T^{-\frac{2}{\alpha}+\e +\rho(2-\alpha+\e)} .
\ea 
Analogously, Lemma \ref{l:moments} yields
\ba
b_T^2  
 \leq \begin{cases}
 \displaystyle C_2 T^{ 2\rho(1-\alpha+\e )}  , \quad  \alpha\in(0,1],\\
 \displaystyle C_2   ,\quad    \alpha\in (1,2).
     \end{cases}
\ea
and hence for each $s\geq 0$ 
\ba
\label{e:b}
\phi_T^2 \cdot a_3(s) \leq  C_3 \max\{1, T^{ 2\rho(1-\alpha+\e )}\}\cdot T^{-\frac{2}{\alpha}+\e}\to 0.
\ea
Finally, for $s\geq 0$
\ba
\label{e:X0}
\phi_T^2 a_1(s)\leq C_4 |X_0|^2 \cdot  T^{-\frac{2}{\alpha}+\e}\to 0\quad \text{a.s.\ as }\ T\to+\infty.
\ea
For any $\rho\in[0,1/\alpha)$ we can choose $\e>0$ sufficiently small such that the bounds in \eqref{e:xi} and \eqref{e:b} and \eqref{e:X0} 
converge to 0 as $T\to\infty$ which gives \eqref{e:XT}.
Integrating these inequalities w.r.t.\ $s\in[0,T]$ results in an additional factor $T$ on the r.h.s.\ of these estimates, and convergence
to $0$ still holds true for $\e>0$ sufficiently small. 
\end{proof}

\begin{lem}
\label{l:Xeta}
For any $\rho>\frac{1}{2\alpha}$ and any $\theta>0$
\ba
\phi_T^2 \int_0^T X^{\eta^T}_{s-}\,\di \eta^T_s\stackrel{\di}{\to} 0,\quad T\to \infty. 
\ea
\end{lem}
\begin{proof}

The Ornstein--Uhlenbeck process $X^{\eta^T}$ as well as its integral w.r.t.\ $\eta^T$ can be written explicitly in the form of sums:
\ba
X^{\eta^T}_t&=\sum_{j=1}^\infty J^T_j \ex^{-\theta (t-\tau_j^T)}\bI_{[\tau_j^T,\infty)}(t),\\
X^{\eta^T}_{\tau^T_k-}&=\sum_{j=1}^{k-1} J^T_j \ex^{-\theta (\tau^T_k-\tau^T_j)},\quad k\geq 1,\\
\int_0^T X^{\eta^T}_{s-}\,\di \eta^T_s&= \sum_{k=1}^{N_T^T} X^{\eta^T}_{\tau^T_k-}J^T_k 
= \sum_{k=1}^{N_T^T}  J_k^T\sum_{j=1}^{k-1} J_j^T\ex^{-\theta (\tau_k^T-\tau_j^T)} .
\ea
As always, we agree that $\sum_{j=k}^m=0$ for $m<k$.

Note that for $N_T^T=0$ and $N_T^T=1$, $\int_0^T X^{\eta^T}_{s-}\,\di \eta^T_s=0$. For $m\geq 2$, on the event $N_T^T=m$ we get 
the estimate
\ba
\label{e:inttt}
\Big|\int_0^T X^{\eta^T}_{s-}\,\di \eta^T_s\Big|
&\leq  \sum_{k=2}^{m}  |J^T_k|\sum_{j=1}^{k-1} |J^T_j|\ex^{-\theta (\tau_k^T-\tau_j^T)}
=\sum_{j=1}^{m-1}\sum_{k=1}^{m-j}        |J_{j+k}^T||J^T_j|\ex^{-\theta (\tau_{j+k}^T-\tau_j^T)}.
\ea
We also take into account that for all $m\geq 2$
\ba
\Law\Big( |J_{j+k}^T ||J_j^T |\ex^{-\theta (\tau_{j+k}^T -\tau_j^T )}\Big| N_T^T=m\Big)  \stackrel{\di}{=}
 \Law\Big( R_T^2\cdot U_T\cdot V_T\cdot \ex^{ -\theta T \sigma_{k}}\Big) ,\\\quad 1\leq j< j+k\leq m,
\ea
where $U_T$, $V_T$ are iid random variables with the probability law
\ba
\P(U_T\geq x)=\frac{H(x R_T)}{H(R_T)},\quad x\geq 1,
\ea
and $\sigma_k$, $k=1,\dots,m-1$, is a $\text{Beta}(k,m-1+k)$-distributed random variable independent of $U_T$ and $V_T$ 
with the probability density \eqref{e:betadensity}.
For each $m\geq 0$ denote for convenience by $\P^{(m)}_T$ the conditional law $\P( \,\cdot\, |N_T^T=m)$. 

For some $\e\in (0,\frac{2-\alpha}{\alpha})$ which will be chosen sufficiently small later, and for each $m\geq 2$ 
define the family of positive weights
\ba
w_{k,m}&= \Big(C(\alpha,\e) \cdot (m-1)\cdot  k^{\frac{2}{\alpha}-\e}\Big)^{-1},\quad k=1,\dots, m-1,\\
\ea
where
\ba
C(\alpha,\e)&=\sum_{k=1}^\infty k^{-\frac{2}{\alpha}+\e}<\infty 
\ea
is the normalizing constant.
With this construction for each $m\geq 2$
\ba
\label{e:w}
\sum_{k=1}^{m-1} \sum_{j=1}^{m-k} w_{k,m}=
\sum_{k=1}^{m-1}(m-k)w_{k,m}=\frac{1}{C(\alpha,\e) } \sum_{k=1}^{m-1} \frac{m-k}{m-1} \cdot  k^{-\frac{2}{\alpha}+\e} \leq 1.
\ea
Let $\gamma>0$. In order to show that the sum \eqref{e:inttt} multiplied by $\phi^2_T$ converges to zero, 
we take into account \eqref{e:w} and write
\ba
\P^{(m)}_T\Big(\phi_T^2 \sum_{j=1}^{m-1}&\sum_{k=1}^{m-j}        |J_{j+k}^T||J_j^T|\ex^{-\theta (\tau_{j+k}-\tau_j)} >\gamma \Big)\\
&\leq 
\P^{(m)}_T\Big(\phi_T^2 \sum_{k=1}^{m-1}\sum_{j=1}^{m-k}  |J_{j+k}^T||J_j^T|\ex^{-\theta (\tau_{j+k}-\tau_j)} >\gamma   
\sum_{k=1}^{m-1}\sum_{j=1}^{m-k} w_{k,m}
\Big)\\
&\leq \sum_{k=1}^{m-1}\sum_{j=1}^{m-k} 
\P \Big(\phi_T^2       R_T^2 U_T V_T\ex^{-\theta T\sigma_{k}} >\gamma w_{k,m} \Big)\\
&=\sum_{k=1}^{m-1 }(m-k)
\P \Big(\phi_T^2       R_T^2 \cdot U_T V_T\ex^{-\theta T\sigma_{k}} >\gamma w_{k,m}  \Big).
\ea
Applying Lemma \ref{l:UV} and the independence of $U_TV_T$ and $\sigma_k$ we obtain for some $\e>0$ 
\ba
p_{k,m}(T)&=\P \Big( \phi_T^2       R_T^2 \cdot U_T V_T\ex^{-\theta T\sigma_{k}}  >\gamma w_{k,m} \Big)\\
&=\frac{m!}{(k-1)!(m-k)!}
\int_0^1  \P\Big( U_TV_T  >\gamma\frac{w_{k,m}}{\phi_T^2 R_T^2}   \cdot \ex^{\theta T u }      \Big)  (1-u)^{m-k}u^{k-1}  \,\di u\\
&\leq  C(\e)  
\frac{m!}{(k-1)!(m-k)!}  \Big(\gamma\frac{w_{k,m}}{\phi_T^2 R_T^2}\Big)^{-\alpha+\e}
\int_0^1 \ex^{-\theta T u(\alpha-\e) }  (1-u)^{m-k}u^{k-1}  \,\di u\\
&\leq C(\e) \frac{m!}{(k-1)!(m-k)!}   \Big(\gamma\frac{w_{k,m}}{\phi_T^2 R_T^2}\Big)^{-\alpha+\e}
\int_0^\infty \ex^{-\theta T u(\alpha-\e) }  u^{k-1}  \,\di u\\
&\leq
C(\e,\alpha,\gamma)
\frac{m!}{(k-1)!(m-k)!}   
\Big( k^{\frac{2}{\alpha}-\e} (m-1)  T^{-\frac{2}{\alpha} + 2\rho+\e}  \Big)^{\alpha-\e}
\frac{(k-1)!}{(\theta  T(\alpha-\e))^k}\\
&\leq C(\e,\alpha,\gamma) \cdot T^{(-\frac{2}{\alpha} + 2\rho+\e)(\alpha-\e)}\cdot
\frac{m!m^2 k^2}{ (m-k)!}   \cdot
\frac{1}{(\theta  T(\alpha-\e))^k},\\
\ea
where we have used the well known relation $\int_0^\infty a u^n \ex^{-a u}\,\di u= n!/a^n$, $a>0$, $n \geq 0$, as well as 
the elementary estimates $(m-1)^{\alpha-\e}\leq m^2$ and $k^{(\frac{2}{\alpha}-\e)(\alpha-\e)}\leq k^2$ which are valid for $\e>0$
and $\alpha\in (0,2)$.

Hence
\ba
\label{e:1}
p(T):=\sum_{m=2}^\infty &\Big[ \P(N_T^T=m)\sum_{k=1}^{m-1}(m-k) p_{k,m}(T)\Big]\\
&=C(\e,\alpha,\gamma)\cdot T^{(-\frac{2}{\alpha} + 2\rho+\e)(\alpha-\e)}\\
&\qquad\times \sum_{m=2}^\infty \Big[ \ex^{-T\cdot H(R_T)}\frac{(T\cdot H(R_T))^m}{m!}  
\sum_{k=1}^{m-1}(m-k) \frac{m!m^2 k^2}{ (m-k)!}   \cdot
\frac{1}{(\theta  T(\alpha-\e))^k} \Big]\\
&= C(\e,\alpha,\gamma)\cdot \ex^{-T\cdot H(R_T)}\cdot T^{(-\frac{2}{\alpha} + 2\rho+\e)(\alpha-\e)} \\
&\qquad \times
\sum_{k=1}^\infty \frac{k^2}{(\theta  T(\alpha-\e))^k}            \sum_{m=k+1}^\infty 
m^2 \frac{(T\cdot H(R_T))^m}{ (m-k-1)!}   .  
\ea
To evaluate the inner sum we use the formula $\sum_{j=0}^\infty (j+k)^2a^j/j!=\ex^a (a^2+2ak+ a+k^2)$  
to obtain
\ba
\label{e:2}
\sum_{m=k+1}^\infty  m^2 \frac{(T\cdot H(R_T))^m}{ (m-k-1)!}   
&=(T\cdot H(R_T))^{k+1}\sum_{j=0}^\infty  (j+k+1)^2 \frac{(T\cdot H(R_T))^j}{ j!}\\
&\leq 3\Big(   (T\cdot H(R_T) )^{k+3} + (k+1)^2 (T\cdot H(R_T) )^{k+1} \Big)   \ex^{T\cdot H(R_T)}.
\ea
Combining \eqref{e:1} and \eqref{e:2}, it is left to estimate two summands. For the first one, we use 
the formula $\sum_{k=1}^\infty k^2 q^k= q(q+1)/(1-q)^3$, $|q|<1$, to get
\ba
S_1=  \sum_{k=1}^\infty \frac{k^2}{(\theta  T(\alpha-\e))^k}  (T\cdot H(R_T) )^{k+3}
\leq C_1 \cdot T^3 \cdot H(R_T)^{4} .
\ea
For the second summand, we use 
the formula $\sum_{k=1}^\infty k^2(k+1)^2 q^k= 4q(q^2+4q+1)/(1-q)^5$, $|q|<1$, to get
\ba
S_2=  
\sum_{k=1}^\infty  \frac{k^2(k+1)^2}{(\theta  T(\alpha-\e))^k} (T\cdot H(R_T) )^{k+1} \leq 
C_2 \cdot T\cdot H(R_T)^2.
\ea
Combining 
\eqref{e:1} with the bounds for $S_1$ and $S_2$ we obtain
\ba
p(T)\leq C\cdot T^{(-\frac{2}{\alpha} + 2\rho+\e)(\alpha-\e)}
\cdot\Big( T^{3- 4\rho(\alpha-\e)}  +  T^{1-2\alpha\rho +\e }\Big).
\ea
Since $\rho>\frac{1}{2\alpha}$, one can choose $\e>0$ sufficiently small
to obtain the limit $p(T)\to 0$, $T\to\infty$. 
\end{proof}

\section{Proofs of the main results\label{s:proofs}}  

\noindent  
\textit{Proof of Theorem \ref{t:X2}.} 
Let $\rho\in (\frac{1}{2\alpha},\frac{1}{\alpha})$ be fixed. With the help of the decomposition \eqref{e:XX} we may write
\ba
\label{e:ab}
\int_0^T X_s^2\,\di s= \int_0^T (X_s^{\eta^T})^2\,\di s + \int_0^T (X_s^T)^2\,\di s 
+ 2 \int_0^T X_s^T \cdot X_s^{\eta^T}\,\di s .
\ea
Then by Lemma \ref{l:OUconv}, $\phi^2_T\int_0^T (X_s^T)^2\,\di s\stackrel{\di}{\to}  0$. Recall that $X^{\eta^T}$ satisfies the SDE
\ba
\di X^{\eta^T}_t=-\theta X^{\eta^T}_t\di t+\di \eta_t^T,\quad  X^{\eta^T}_0=0.
\ea
The It\^o formula applied to the process $X^{\eta^T}$ yields
\ba
\label{e:aa}
\big(X^{\eta^T}_T\big)^2= -2\theta \int_0^T  \big( X^{\eta^T}_{s} \big)^2\,\di s 
+2 \int_0^T   X^{\eta^T}_{s-} \,\di \eta^T_s + [\eta^T]_T.
\ea
The decomposition \eqref{e:XX} implies that $(X^{\eta^T}_T)^2\leq 2 X_T^2+ 2 (X^{T}_T)^2$.
Since for $\theta>0$ the process $X$ has an invariant distribution (see, e.g.\ \cite[Theorem 17.5]{Sato-99}), we get that
$\phi^2_T X_T^2\to 0$ in law. 
On the other hand, $\phi^2_T(X^{T}_T)^2\to 0$ in law by Lemma \ref{l:OUconv}.  
Therefore, Lemmas \ref{l:Slaplace}, \ref{l:Xeta} and \eqref{e:aa} yield
\ba
\phi_T^2 \int_0^T  \big( X^{\eta^T}_{s} \big)^2\,\di s \stackrel{\di}{\to} \frac{\mathcal S^{(\alpha/2)}}{2\theta},\quad T\to\infty.
\ea
Eventually, the last integral in \eqref{e:ab} multiplied by $\phi_T^2$ converges to 0 by the Cauchy--Schwarz inequality. \hfill $\Box$

\bigskip

\noindent
\textit{Proof of Corollary \ref{t:XWX2}.}
The decomposition 
\ba
X_t=X_t^T+X_t^{\eta^T}= X_0\ex^{-\theta t} + b_T\frac{1-\ex^{-\theta s}}{2\theta} + \int_0^t \ex^{-\theta(t-s)}\,\di \xi^T_s +X^{\eta^T}_t
\ea
allows us to write
\ba
&\int_0^T X_s\,\di W_s=\int_0^T X_s^T\,\di W_s+\int_0^T X_s^{\eta^T}\,\di W_s
\ea
as well as \eqref{e:ab}.
It is easy to check that $\phi_T\int_0^T X_s^T\,\di W_s\stackrel{\di}{\to} 0$. Indeed,
due to the independence of $X_0$ and $W$
\ba
\phi_T \int_0^T X_0\ex^{-\theta s}\,\di W_s = \phi_T \cdot X_0 \cdot \int_0^T\ex^{-\theta s}\,\di W_s \to 0\quad \text{a.s.}
\ea
and obviously by Lemma \ref{l:Xeta}
\ba
\phi_T b_T \int_0^T  (1-\ex^{-\theta s})  \,\di W_s \stackrel{\di}{\to} 0.
\ea
Finally
by the estimate \eqref{e:xi} of Lemma \ref{l:OUconv}
\ba
\E\Big[\phi_T \int_0^T \Big(X^{T}_s&- X_0\ex^{-\theta s} -b_T\frac{1-\ex^{-\theta s}}{2\theta}    \Big)\,\di W_s\Big]^2
=\phi_T^2 \E\Big[\int_0^T \int_0^s \ex^{-\theta(s-r)}\,\di \xi^T_r  \,\di W_s\Big]^2\\
&=\phi_T^2 \int_0^T \E\Big[\int_0^s \ex^{-\theta(s-r)}\,\di \xi^T_r\Big]^2  \,\di s\\
&=\phi_T^2 \int_0^T \int_0^s \ex^{-2\theta(s-r)}\,\di r  \,\di s\cdot \Big(\sigma^2 +\int_{|z|\leq R_T}z^2\nu(\di z)\Big)
 \to 0,\quad T \to 0.
\ea
Taking into account the argument in the proof of Theorem \ref{t:X2}, we conclude
that it is sufficient to consider the limiting behaviour of the pair 
$\big( \phi_T \int_0^T X^{\eta^T}_s\,\di W_s,  \phi_T^2 \int_0^T (X^{\eta^T}_s)^2 \,\di s  \big)$.

The processes $\eta^T$ and $W$ are independent and
\ba
M^T_t=\int_0^t X^{\eta^T}_s\,\di W_s,\quad t\geq 0,
\ea
is a continuous local martingale with the angle bracket
\ba
\langle M^T\rangle_t=\int_0^t \big(X_s^{\eta^T}\big)^2\,\di s,
\ea
which is independent of $W$. 
Then for $u,v\in\bR$ we  get
\ba
\E \exp\Big( \i u \phi_T M^T_T &+\i v \phi_T^2 \langle M^T\rangle_T \Big) =
\E \Big[ \E\Big[ \exp\Big( \i u \phi_T M^T_T +\i v \phi_T^2 \langle M^T\rangle_T \Big) \Big| \rF^{\eta^T}_T \Big]\\
&=\E\Big[\exp\Big( \i v \phi_T^2 \langle M^T\rangle_T \Big) \E \Big [ \exp\Big( \i u \phi_T M^T_T    \Big)\Big|\rF^{\eta^T}_T \Big]\Big] \\
&=\E \Big[\exp\Big( \i v \phi_T^2 \langle M^T\rangle_T \Big) \exp\Big( -\frac{u^2}{2}\phi_T^2 \langle M^T\rangle_T    \Big) \Big] \\
&=\E\exp\Big( \big(\i v -\frac{u^2}{2}\big) \phi_T^2 \langle M^T\rangle_T    \Big)\\
&\to\E\exp\Big( \big(\i v -\frac{u^2}{2}\big)   \frac{\mathcal S^{(\alpha/2)}}{2\theta_0}  \Big) \\
&=\E \exp\Big( \i u \cN \sqrt{  \frac{\mathcal S^{(\alpha/2)}}{2\theta_0} } +\i v \frac{\mathcal S^{(\alpha/2)}}{2\theta_0} \Big).
\ea
 \hfill $\Box$

\bigskip

\noindent
\textit{Proof of Theorem \ref{t:main}.} The statement of the theorem follows immediately from Proposition \ref{p:Lratio} and 
Corollary \ref{t:XWX2}.
 \hfill $\Box$

\bigskip

\noindent
\textit{Proof of Corollary \ref{c:2}.} 
The relation \eqref{e:MLE} follows from Proposition \ref{p:Lratio}. Due to the linear-quadratic form of the likelihood ratio,
the maximum likelihood estimator coincides with the so-called central sequence. 
This implies the asymptotic efficiency in the aforementioned sense.
The limit \eqref{e:AB} follows from Corollary \ref{t:XWX2}.

   \hfill $\Box$
   
  \bigskip

\noindent
\textit{Proof of Remark \ref{r:D}.} For $x>0$
\ba
\P\Big(\frac{|\cN|}{\sqrt{\mathcal S^{(\alpha/2)}}}> x\Big) &\leq \P\Big(\mathcal S^{(\alpha/2)}
\leq x^{\alpha-2}\Big) + \P(|\cN|> x^{\alpha/2})=p_1(x)+p_2(x).
\ea
By the well known property of the Gaussian distribution 
\ba
p_2(x)\leq \sqrt{\frac{2}{\pi}}\frac{\ex^{-x^{\alpha}/2}}{x^{\alpha/2}}.
\ea
To estimate $p_1(x)$, we apply the exponential Chebyshev inequality to get
\ba
p_1(x)=\P\Big(\mathcal S^{(\alpha/2)}\leq x^{\alpha-2}\Big) &\leq \inf_{\lambda>0 } \ex^{\lambda x^{\alpha-2}} \E \ex^{-\lambda S^{(\alpha/2)}}\\
&=\inf_{\lambda>0 } \ex^{\lambda x^{\alpha-2} -\Gamma(1-\frac{\alpha}{2})\lambda^{\alpha/2}   }\leq
\exp\Big( -C(\alpha) x^{\alpha}  \Big)
\ea
for some $C(\alpha)>0$. Hence the estimate \eqref{e:tail} follows.


\begin{thebibliography}{32}
\providecommand{\natexlab}[1]{#1}
\providecommand{\url}[1]{\texttt{#1}}
\expandafter\ifx\csname urlstyle\endcsname\relax
  \providecommand{\doi}[1]{doi: #1}\else
  \providecommand{\doi}{doi: \begingroup \urlstyle{rm}\Url}\fi

\bibitem[Applebaum(2009)]{Applebaum-09}
D.~Applebaum.
\newblock \emph{{{L{\'e}}vy Processes and Stochastic Calculus}}, volume 116 of
  \emph{{Cambridge Studies in Advanced Mathematics}}.
\newblock Cambridge University Press, Cambridge, second edition, 2009.

\bibitem[Balakrishnan and Nevzorov(2003)]{BalNev-03}
N.~Balakrishnan and V.~B. Nevzorov.
\newblock \emph{{A Primer on Statistical Distributions}}.
\newblock John Wiley \& Sons, Hoboken, NJ, 2003.

\bibitem[Bingham et~al.(1987)Bingham, Goldie, and Teugels]{BinghamGT-87}
N.~H. Bingham, C.~M. Goldie, and J.~L. Teugels.
\newblock \emph{{Regular Variation}}, volume~27 of \emph{{Encyclopedia of
  Mathematics and its applications}}.
\newblock Cambridge University Press, 1987.

\bibitem[Cl{\'e}ment and Gloter(2015)]{clement2015local}
E.~Cl{\'e}ment and A.~Gloter.
\newblock {Local asymptotic mixed normality property for discretely observed
  stochastic differential equations driven by stable L{\'e}vy processes}.
\newblock \emph{Stochastic Processes and their Applications}, 125\penalty0
  (6):\penalty0 2316--2352, 2015.

\bibitem[Cl{\'e}ment and Gloter(2019)]{ClementG-19}
E.~Cl{\'e}ment and A.~Gloter.
\newblock {Joint estimation for SDE driven by locally stable L{\'e}vy
  processes}.
\newblock {Preprint}, 2019.
\newblock {HAL Id: hal-02125428}.

\bibitem[Cl{\'e}ment et~al.(2019)Cl{\'e}ment, Gloter, and
  Nguyen]{clement2019lamn}
E.~Cl{\'e}ment, A.~Gloter, and H.~Nguyen.
\newblock {LAMN property for the drift and volatility parameters of a SDE
  driven by a stable L{\'e}vy process}.
\newblock \emph{ESAIM: Probability and Statistics}, 23:\penalty0 136--175,
  2019.

\bibitem[Gloter et~al.(2018)Gloter, Loukianova, and Mai]{gloter2018jump}
A.~Gloter, D.~Loukianova, and H.~Mai.
\newblock {Jump filtering and efficient drift estimation for L{\'e}vy-driven
  SDEs}.
\newblock \emph{The Annals of Statistics}, 46\penalty0 (4):\penalty0
  1445--1480, 2018.

\bibitem[H{\"o}pfner(2014)]{hoepfner2014asymptotic}
R.~H{\"o}pfner.
\newblock \emph{{Asymptotic Statistics: With a View to Stochastic Processes}}.
\newblock Walter de Gruyter, Berlin, 2014.

\bibitem[Hu and Long(2007)]{hu2007parameter}
Y.~Hu and H.~Long.
\newblock {Parameter estimation for Ornstein-Uhlenbeck processes driven by
  $\alpha$-stable L{\'e}vy motions}.
\newblock \emph{Communications on Stochastic Analysis}, 1\penalty0
  (2):\penalty0 1, 2007.

\bibitem[Hu and Long(2009{\natexlab{a}})]{hu2009least}
Y.~Hu and H.~Long.
\newblock {Least squares estimator for Ornstein--Uhlenbeck processes driven by
  $\alpha$-stable motions}.
\newblock \emph{Stochastic Processes and their Applications}, 119\penalty0
  (8):\penalty0 2465--2480, 2009{\natexlab{a}}.

\bibitem[Hu and Long(2009{\natexlab{b}})]{yaozhong2009singularity}
Y.~Hu and H.~Long.
\newblock On the singularity of least squares estimator for mean-reverting
  $\alpha$-stable motions.
\newblock \emph{Acta Mathematica Scientia}, 29\penalty0 (3):\penalty0 599--608,
  2009{\natexlab{b}}.

\bibitem[Ivanenko and Kulik(2014)]{ivanenko2014lan}
D.~Ivanenko and A.~Kulik.
\newblock {LAN} property for discretely observed solutions to {L}{\'e}vy driven
  {SDE}'s.
\newblock \emph{Modern Stochastics: Theory and Applications}, 1\penalty0
  (1):\penalty0 33--47, 2014.

\bibitem[Jacod and Shiryaev(2003)]{JacodS-03}
J.~Jacod and A.~N. Shiryaev.
\newblock \emph{{Limit Theorems for Stochastic Processes}}, volume 288 of
  \emph{{Grundlehren der Mathematischen Wissenschaften}}.
\newblock Springer, second edition, 2003.

\bibitem[Jessen and Mikosch(2006)]{jessen2006regularly}
A.~H. Jessen and T.~Mikosch.
\newblock Regularly varying functions.
\newblock \emph{Publications de L'Institut Mathematique. Nouvelle s\'erie},
  80(94):\penalty0 171--192, 2006.

\bibitem[Kawai(2013)]{kawai2013local}
R.~Kawai.
\newblock {Local asymptotic normality property for Ornstein--Uhlenbeck
  processes with jumps under discrete sampling}.
\newblock \emph{Journal of Theoretical Probability}, 26\penalty0 (4):\penalty0
  932--967, 2013.

\bibitem[Kohatsu-Higa et~al.(2017)Kohatsu-Higa, Nualart, and
  Tran]{kohatsu2017lan}
A.~Kohatsu-Higa, E.~Nualart, and N.~K. Tran.
\newblock {LAN property for an ergodic diffusion with jumps}.
\newblock \emph{Statistics: A Journal of Theoretical and Applied Statistics},
  51\penalty0 (2):\penalty0 419--454, 2017.

\bibitem[K{\"u}chler and S\o{}rensen(1997)]{KuechlerS-97}
U.~K{\"u}chler and M.~S\o{}rensen.
\newblock \emph{{Exponential Families of Stochastic Processes}}.
\newblock {Springer Series in Statistics}. Springer-Verlag, New York, 1997.

\bibitem[Le~Cam and Yang(2000)]{lecam2000}
L.~Le~Cam and G.~L. Yang.
\newblock \emph{Asymptotics in Statistics: Some Basic Concepts}.
\newblock Springer Series in Statistics. Springer-Verlag, New York, second
  edition, 2000.

\bibitem[Long(2009)]{long2009least}
H.~Long.
\newblock {Least squares estimator for discretely observed Ornstein--Uhlenbeck
  processes with small L{\'e}vy noises}.
\newblock \emph{Statistics \& Probability Letters}, 79\penalty0 (19):\penalty0
  2076--2085, 2009.

\bibitem[Mai(2012)]{mai2012diss}
H.~Mai.
\newblock \emph{{Drift Estimation for Jump Diffusions}}.
\newblock {PhD thesis}, Humboldt-Universit{\"a}t zu Berlin, 2012.
\newblock {https://doi.org/10.18452/16590}.

\bibitem[Mai(2014)]{Mai-14}
H.~Mai.
\newblock {Efficient maximum likelihood estimation for L{\'e}vy-driven
  Ornstein--Uhlenbeck processes}.
\newblock \emph{Bernoulli}, 20\penalty0 (2):\penalty0 919--957, 2014.

\bibitem[Masuda(2013)]{masuda2013convergence}
H.~Masuda.
\newblock {Convergence of Gaussian quasi-likelihood random fields for ergodic
  L{\'e}vy driven SDE observed at high frequency}.
\newblock \emph{The Annals of Statistics}, 41\penalty0 (3):\penalty0
  1593--1641, 2013.

\bibitem[Masuda(2015)]{masuda2015parametric}
H.~Masuda.
\newblock {Parametric estimation of L{\'e}vy processes}.
\newblock In \emph{L{\'e}vy Matters IV}, volume 2128 of \emph{Lecture Notes in
  Mathematics}, pages 179--286. Springer, 2015.

\bibitem[Masuda(2019)]{masuda2019non}
H.~Masuda.
\newblock {Non-Gaussian quasi-likelihood estimation of SDE driven by locally
  stable L{\'e}vy process}.
\newblock \emph{Stochastic Processes and their Applications}, 129\penalty0
  (3):\penalty0 1013--1059, 2019.

\bibitem[Nguyen(2018)]{nguyen2018estimation}
T.~T.~H. Nguyen.
\newblock \emph{Estimation of the Jump Processes}.
\newblock {PhD thesis}, Universit{\'e} {P}aris-{E}st, 2018.
\newblock {HAL Id: tel-02127797}.

\bibitem[Resnick(2007)]{resnick2007heavy}
S.~I. Resnick.
\newblock \emph{{Heavy-Tail Phenomena: Probabilistic and Statistical
  Modeling}}.
\newblock Springer Series in Operations Research and Financial Engineering.
  Springer Science \& Business Media, New York, 2007.

\bibitem[Rosi{\'n}ski and Woyczy{\'n}ski(1987)]{rosinski1987multilinear}
J.~Rosi{\'n}ski and W.~A. Woyczy{\'n}ski.
\newblock {Multilinear forms in Pareto-like random variables and product random
  measures}.
\newblock In \emph{Colloquium Mathematicum}, volume~1, pages 303--313, 1987.

\bibitem[Sato(1999)]{Sato-99}
K.~Sato.
\newblock \emph{{L{\'e}vy Processes and Infinitely Divisible Distributions}},
  volume~68 of \emph{{Cambridge Studies in Advanced Mathematics}}.
\newblock Cambridge University Press, Cambridge, 1999.

\bibitem[S\o{}rensen(1991)]{Soerensen-91}
M.~S\o{}rensen.
\newblock {Likelihood methods for diffusions with jumps}.
\newblock In N.~U. Prabhu and I.~V. Basawa, editors, \emph{{Statistical
  Inference in Stochastic Processes}}, pages 67--105. Marcel Dekker, Inc., New
  York, 1991.

\bibitem[Tran(2017)]{tran2017lan}
N.~K. Tran.
\newblock {LAN property for an ergodic Ornstein--Uhlenbeck process with Poisson
  jumps}.
\newblock \emph{Communications in Statistics-Theory and Methods}, 46\penalty0
  (16):\penalty0 7942--7968, 2017.

\bibitem[Uehara(2019)]{uehara2019statistical}
Y.~Uehara.
\newblock {Statistical inference for misspecified ergodic L{\'e}vy driven
  stochastic differential equation models}.
\newblock \emph{Stochastic Processes and their Applications}, 129\penalty0
  (10):\penalty0 4051--4081, 2019.

\bibitem[Zhang and Zhang(2013)]{zhang2013least}
S.~Zhang and X.~Zhang.
\newblock {A least squares estimator for discretely observed
  Ornstein--Uhlenbeck processes driven by symmetric $\alpha$-stable motions}.
\newblock \emph{Annals of the Institute of Statistical Mathematics},
  65\penalty0 (1):\penalty0 89--103, 2013.

\end{thebibliography}
%

\end{document}